\documentclass[a4paper,11pt]{amsart}

\usepackage{amsmath}
\usepackage{amssymb}
\usepackage{amsthm} 
\usepackage{amsfonts,amsbsy}
\usepackage{hyperref}
\usepackage{enumerate}

\usepackage[margin=3.5cm]{geometry}

\theoremstyle{plain}
\newtheorem{Thm}{Theorem}[section]
\newtheorem{Lemm}[Thm]{Lemma}

\newtheorem{Prop}[Thm]{Proposition}

\newtheorem{Conj}[Thm]{Conjecture}
\newtheorem*{Claim}{Claim}
\newtheorem*{uThm}{Theorem}

\theoremstyle{definition}
\newtheorem{Def}[Thm]{Definition}
\newtheorem{Rem}[Thm]{Remark}

\newtheorem{Nota}[Thm]{Notation}

\newcommand{\Z}{\mathbb{Z}} 

\newcommand{\Q}{\mathbb{Q}} 
\newcommand{\R}{\mathbb{R}}
\newcommand{\C}{\mathbb{C}}

\newcommand{\expE}{\exp_{\Elg}}
\newcommand{\w}{\omega}
 
\newcommand{\alg}{^{\text{alg}}}

\newcommand{\X}{\mathcal{X}}
\newcommand{\Y}{\mathcal{Y}}

\newcommand{\K}{\mathcal{K}}

\newcommand{\RR}{\mathcal{R}}

\newcommand{\CC}{\mathcal{C}}
\newcommand{\SW}{\mathcal{S}}

\newcommand{\E}{E}
\newcommand{\Elg}{\mathbb{E}}
\newcommand{\PP}{\mathbb{P}}
\newcommand{\G}{\mathbb{G}}
\newcommand{\Alg}{\mathbb{E}}
\newcommand{\A}{\mathcal{A}}

\newcommand{\spanEndA}{\operatorname{span}_{\End(\Alg)}}
\newcommand{\spank}{\operatorname{span}}

\newcommand{\Tor}{\operatorname{Tor}}

\newcommand{\Real}{\operatorname{Re}}
\newcommand{\Imag}{\operatorname{Im}}

\newcommand{\acl}{\operatorname{acl}}
\newcommand{\dcl}{\operatorname{dcl}}

\newcommand{\ex}{\operatorname{ex}}
\newcommand{\ld}{\operatorname{lin.d.}}
\newcommand{\md}{\operatorname{mult.d.}}

\newcommand{\trd}{\operatorname{tr.d.}}

\newcommand{\spanQ}{\operatorname{span}_{\Q}}

\newcommand{\cl}{\operatorname{cl}}
\newcommand{\scl}{\operatorname{scl}}
\newcommand{\End}{\operatorname{End}}

\newcommand{\GL}{\operatorname{GL}}

\newcommand{\dd}{\operatorname{d}}

\newcommand{\fin}{_\text{fin}} 

\newcommand{\eps}{\epsilon} 
\newcommand{\subsetfin}{\subset_{\fin}}

\newlength {\xxxIndep}
\newlength {\yyyIndep}

\begin{document}

\title{Natural models of theories of green points}
\author{Juan Diego Caycedo}
\address{Mathematisches Institut, Albert-Ludwigs-Universit\"at Freiburg, Eckerstr. 1, 79104 Freiburg, Deutschland.}
\email{juan-diego.caycedo@math.uni-freiburg.de}
\author{Boris Zilber}
\address{Mathematical Institute, University of Oxford, Woodstock Road, Andrew Wiles Building, Radcliffe Observatory Quarter, Oxford, OX2 6GG, United Kingdom.}
\email{zilber@maths.ox.ac.uk}

\date{January 10, 2014}

\begin{abstract}
We explicitly present expansions of the complex field which are models of the theories of green points in the multiplicative group case and in the case of an elliptic curve without complex multiplication defined over $\R$. In fact, in both cases we give families of structures depending on parameters and prove that they are all models of the theories, provided certain instances of Schanuel's conjecture or an analogous conjecture for the exponential map of the elliptic curve hold. In the multiplicative group case, however, the results are unconditional for generic choices of the parameters.
\end{abstract}

\maketitle


\bibliographystyle{alpha}

\section{Introduction}

The object of this paper is to explicitly give models on the complex numbers for the theories of green points constructed in \cite{CayGreenTh}. 
The models of the theories of green points are expansions of the natural algebraic structure on $\Alg(K)$, where $K$ is an algebraically closed field of characteristic zero and $\Alg$ is the multiplicative group or on an elliptic curve defined over $K$, by a predicate for a divisible (non-algebraic) subgroup which is generic with respect to a certain predimension function. Elements of the subgroup are called \emph{green points}, by a convention introduced by Poizat. Indeed, the case where $\Alg$ is the multiplicative group and the subgroup is required to be torsion-free, corresponds to the theory of fields with green points constructed by Poizat. The theories are $\omega$-stable of Morley rank $\omega \cdot 2$.

The present work corrects and extends a result of \cite{ZBicol} which gives a model of Poizat's theory of fields with green points on the complex numbers, under the assumption that Schanuel's conjecture holds. Here too, we will in general have to assume certain instances of Schanuel's conjecture in the multiplicative group case and of an analogous conjecture in the case of an elliptic curve. For the multiplicative group, our results will be unconditional in generic cases.


The main results of this paper are presented in Sections~\ref{section:green-model}~and~\ref{section:elliptic-model}. 

Section~\ref{section:green-model} deals with the multiplicative group case. The following is a simplified statement of the main result of that section, Theorem~\ref{theorem:green}:

\begin{uThm}
Let $\eps = 1 + \beta i$, with $\beta$ a non-zero real number, and let $Q$ be a non-trivial divisible subgroup of $(\R,+)$ of finite rank. Let
\[
G = \exp(\eps \R + Q).
\]
Assume the Schanuel Conjecture for raising to powers in $K = \Q(\beta i)$. 
Then the structure $(\C^*,G)$ can be expanded by constants to a model of a theory of green points. In particular, $(\C^*,G)$ is $\omega$-stable. 
\end{uThm}

The Schanuel Conjecture for raising to powers in a subfield $K$ of $\C$ is an (unproven) consequence of Schanuel's conjecture that will be discussed in Section~\ref{sec:prelim}. In the cases where $\beta$ is generic in the o-minimal structure $\R_{\exp}$, the Schanuel Conjecture for raising to powers in $K = \Q(\beta i)$ is known to hold by a theorem of Bays, Kirby and Wilkie (\cite{BKW}). The above result is therefore unconditional in those cases.

In Subsection~\ref{sec-eme}, we derive an analogous result for the theories of emerald points constructed in \cite[Section 5]{CayGreenTh}. These are variations of the theories of green points where the distinguished subgroup is elementarily equivalent to the additive group of the integers. They are superstable, non-$\omega$-stable, of U-rank $\omega \cdot 2$.

The elliptic curve case is treated in Section~\ref{section:elliptic-model}. The main theorem of Section~\ref{section:elliptic-model} is Theorem~\ref{theorem:elliptic}, which we now state, again in a simpler version.

\begin{uThm}
Let $\Elg$ be an elliptic curve without complex multiplication defined over $\C$ and let $\E = \Elg(\C)$. Assume the corresponding lattice $\Lambda$ has the form $\Z + \tau\Z$ and is invariant under complex conjugation. 

Let $\eps = 1+ \beta i$, with $\beta$ a non-zero real number, be such that $\eps \R \cap \Lambda = \{ 0 \}$. Put $G = \expE(\eps \R)$. 

Assume the Weak Elliptic Schanuel Conjecture for raising to powers in $K := \Q(\beta i)$ (wESC$_K$) holds for $\Elg$.
Then the structure $(\E,G)$ can be expanded by constants to a model of a theory of green points. In particular, $(\E,G)$ is $\omega$-stable.
\end{uThm}

Let us note that the assumption that the lattice $\Lambda$ has the form $\Z + \tau\Z$ can always be made to hold by passing to an isomorphic elliptic curve. That $\Lambda$ is invariant under complex conjugation is a restrictive assumption. It holds, however, whenever $\Elg$ is defined over $\R$. The Weak Elliptic Schanuel Conjecture for raising to powers in $K$ will be introduced in Section~\ref{sec:prelim}.


The above results fit into the programme, first outlined in \cite{ZPseudoAnalytic}, of finding mathematically natural models for model-theoretically well-behaved theories. They also provide new examples of (explicitly given) stable expansions of the complex field. Most known examples of such structures are covered by the theorems on expansions by \emph{small sets} in \cite{ZieCas} and the green subgroups are not small.

Our proofs of the main theorems follow the same strategy as in \cite{ZBicol}. Ax's theorem on a differential version of Schanuel's conjecture from \cite{Ax71} plays a key role and some geometric arguments combine elements of complex analytic geometry and o-minimality.

The paper begins with two preliminary sections: In Section~\ref{sec:axioms}, the axioms of the theories of green points are recalled. Section~\ref{sec:prelim} contains necessary preliminaries on structures on the complex numbers related to exponentiation.

The research presented here was part of the D.Phil. thesis of the first author, written under the supervision of the second author at the University of Oxford. It was funded by the Marie Curie Research Training Network MODNET.

\section{The theories} \label{sec:axioms}


We shall now introduce several basic notions and state the conditions that a structure must satisfy to be a model of one of the theories of green points constructed in \cite{CayGreenTh}. 


In this section, let $\Alg$ be the multiplicative group or an elliptic curve over a field $k_0$ of characteristic 0. We use additive notation for the group operation on $\Alg$. 

Let $L_{\Alg}$ be the first-order language consisting of an $n$-ary predicate for each subvariety of $\Alg^n$ defined over $k_0$, $n \geq 1$.

For each algebraically closed field $K$ extending $k_0$, we have a natural $L_{\Alg}$-structure on $A:=\Alg(K)$, namely:
\[
(\Alg(K), (W(K))_{W \in L_{\Alg}}).
\]
All these structures satisfy the same complete $L_{\Alg}$-theory $T_{\Alg}$ and every model of $T_{\Alg}$ is of this form. 

Also, $A$ is an $\End(\Alg)$-module. The dimension function on $A$ given by the $\End(\Alg)$-linear dimension, will be denoted by $\ld_{\End(\Alg)}$, or simply by $\ld$. We use $\langle Y \rangle$ or $\spanEndA(Y)$ to denote the $\End(\Alg)$-span of a subset $Y$ of $A$.  Since $K$ is algebraically closed, $A$ is divisible. The ring $\End(\Alg)$ is an integral domain and $k_{\Alg} := \End(\Alg) \otimes_{\Z} \Q$ is its fraction field. The quotient $A/\Tor(A)$ is a $k_{\Alg}$-vector space and for every $Y \subset A$, $\ld_{\End(\Alg)}(Y)$ equals the $k_{\Alg}$-linear dimension of $\phi(Z)$ in $A/\Tor(A)$, where $\phi:A \to A/\Tor(A)$ is the quotient map. The pregeometry on $A/\Tor(A)$ given by the $k_{\Alg}$-span induces a pregeometry on $A$ that we shall denote by $\spank$; this means that for $Y \subset A$, $\spank(Y) = \phi^{-1}(\operatorname{span}_{k_{\Alg}} (\phi(Y)))$.

The theory $T_{\Alg}$ is strongly minimal. The $\acl$-dimension of a tuple $b \in A$ equals the transcendence degree over $k_0$ of any normalised representation of $b$ in homogeneous coordinates, which we shall denote by $\trd(b/k_0)$ or $\trd_{k_0}(b)$. We write $\trd$ for $\trd_{\Q}$.



Let $L = L_{\Alg} \cup \{G\}$ be the expansion of the language $L_{\Alg}$ by a unary predicate $G$. 

Let $\CC$ be the class of all $L$-structures $\A = (A,G)$ where $A$ is a model of $T_{\Alg}$ and $G$ is a divisible $\End(\Alg)$-submodule of $A$. 

Following a convention introduced by Poizat, given an $L$-structure $\A = (A,G)$ in $\CC$, we call the elements of $G$ \emph{green points} and the elements of $A \setminus G$ \emph{white points}.


Consider the \emph{predimension function} $\delta$ defined on the finite dimensional $\spank$-closed subsets $X$ of any structure $\A \in\CC$ by
\[
\delta(X) = 2 \trd_{k_0}(X) - \ld(X \cap G).
\]
Also, for any $\spank$-closed subset $Y$ of $A$, the \emph{localisation of $\delta$ at $Y$}, $\delta_Y$, is the function given by 
\[
\delta_Y(X) = 2 \trd_{k_0}(X/Y) - \ld((X+Y)\cap G/Y \cap G),
\]
for any finite dimensional $\spank$-closed set $X$. We also write $\delta(X/Y)$ for the value $\delta_Y(X)$, and call it the \emph{predimension of $X$ over $Y$}. 

Associated to the predimension function $\delta$ we have the notion of strong sets. 
A $\spank$-closed subset $Y$ of $A$ is \emph{strong} in the structure $\A$ (with respect to $\delta$), if for every finite dimensional $\spank$-closed subset $X$ of $A$ we have 
\[
\delta(X/Y) \geq 0.
\]
An arbitrary subset $Y$ of $A$ is said to be strong in $\A$ if $\spank(Y)$ is strong in $\A$ in the above sense.
If $\Y$ is a substructure of $\A$ and its domain $Y$ is strong in $\A$, then we say $\Y$ is a \emph{strong substructure} of $\A$, and that $\A$ is a \emph{strong extension} of $\Y$.


Let us fix a substructure $\X_0$ of a structure in $\CC$ whose domain is a finite dimensional $\spank$-closed set. Let $L_{X_0}$ denote the expansion of the language $L$ by constants for the elements of $X_0$. 

Let $\CC_0$ be the class of all $L_{X_0}$-structures $\A_{X_0}$ such that the $L$-reduct $\A$ is in the class $\CC$ and the interpretation of the constants gives a strong embedding of $\X_0$ into $\A$. 
For $\A \in \CC_0$, we identify $\X_0$ with the strong substructure of $A$ consisting of the interpretations of the constants. With this convention in place, we may omit any explicit mention of the constants in the notation for a structure in $\CC_0$, writing simply $\A$ instead of $\A_{X_0}$.



Let us recall two definitions from \cite{CayGreenTh}.

\begin{Def} \label{def-rotund}
An irreducible subvariety $W$ of $A^n$ is said to be \emph{rotund} if for every $k \times n$-matrix $M$ with entries in $\End(\Alg)$ of rank $k$, the dimension of the constructible set $M \cdot W := \{ M \cdot y : y \in W\}$ is at least $\frac{k}{2}$.
\end{Def}

It is worth noting that for any subvariety $W$ of $A^n$ and any $C\subset A$ such that $W$ is defined over $k_0(C)$, if $b$ is a generic point of $W$ over $k_0(C)$, then: $W$ is rotund if and only if for every $k \times n$-matrix $M$ with entries in $\End(\Alg)$ of rank $k$,
\[
\trd(M \cdot b/k_0(C)) \geq \frac{k}{2}.
\]

\begin{Def}
A structure $\A = (A,G) \in \CC$ is said to have the \emph{EC-property} if for every even $n\geq 1$ and every rotund subvariety $W$ of $A^n$ of dimension $\frac{n}{2}$, the intersection $W \cap G^n$ is Zariski dense in $W$; i.e. for every proper subvariety $W'$ of $W$ the intersection $(W \setminus W') \cap G^n$ is non-empty.
\end{Def}




It was shown in \cite{CayGreenTh} (see, in particular, Lemmas~3.8~and~3.17) that there is a complete $L_{X_0}$-theory $T_{X_0}$ whose models are precisely the structures $(A,G)_{X_0}$ such that
\begin{enumerate}
\item $(A,G)_{X_0}$ is in $\CC_0$ (i.e. $X_0$ is strong in $(A,G)$), and
\item $(A,G)$ has the EC-property.
\end{enumerate}
The \emph{theories of green points} are the theories obtained in this way.

\section{Some preliminaries} \label{sec:prelim}

This section gathers several definitions and results about structures on the complex numbers related to exponentiation. These will be applied in the later sections.

\subsection{Exponentiation and raising to powers}
\label{chapter:related}

\subsubsection{Exponentiation}

Let $\C_{\exp} = (\C,+,\cdot, \exp)$ be the expansion of the complex field by the exponential function. The \emph{Schanuel Conjecture} from transcendental number theory, which we state below, can be regarded as a statement about $\C_{\exp}$.

\begin{Conj}[The Schanuel Conjecture (SC)]
For every $n$ and every $\Q$-linearly independent tuple $x \in \C^n$,
\[
\trd(x \exp x) \geq n.
\]
\end{Conj}

The predimension function $\delta_{\exp}$ is defined on any tuple $x \subset \C$ by
\[
\delta_{\exp}(x) := \trd(x \exp x) - \ld_{\Q}(x).
\]
The Schanuel conjecture is equivalent to the statement that for every $x \subset \C$, $\delta_{\exp}(x) \geq 0$. Therefore, if the SC holds, then $\delta_{\exp}$ is a proper predimension function on $\C_{\exp}$.


In \cite{ZPseudoExp}, a model-theoretic study of the structure $\C_{\exp}$ is carried out using the predimension function $\delta_{\exp}$. Here we shall only need one aspect of that work, namely Zilber's proof that the pregeometry associated to $\delta_{\exp}$ has the \emph{countable closure property} (\cite[Lemma 5.12]{ZPseudoExp}). Versions of this fact will be essential in our arguments in sections \ref{section:green-model} and \ref{section:elliptic-model}. We include the proof, in slightly greater detail than in \cite{ZPseudoExp}.


Let us recall a general fact that applies to each of the predimension functions $\delta$ on $\C$ considered in this paper (for more details we refer to \cite[Definition 2.24 and Remark 2.25]{CayGreenTh}). If $\delta$ is non-negative, we have an associated dimension function $\dd$, defined for all finite $X \subset A$ by 
\[
\dd(X) = \min \{ \delta(X') : X \subset X' \subsetfin \C\}.
\]
and a corresponding pregeometry $\cl$ on the set $\{ x \in \C : \dd(x) \leq 1 \}$ (with dimension function $\dd$) given by: for $X \subsetfin \C$ and $x_0 \in \C$, 
\[
x_0 \in \cl_{\dd}(X) \iff \dd(x_0/X) = 0.
\]
Let us note the following alternative formulations of the above:
\begin{align*}
     &x_0 \in \cl_{\dd}(X)\\
\iff &\dd(x_0/X) = 0\\
\iff &\dd(x_0 X) = \dd(X)\\
\iff &\text{ there exists a tuple } x \supset x_0 \text{ such that } \delta(x/\scl(X)) = 0\\
\iff &x_0 \in \scl(X) \text{ or } \text{there exists a tuple } x \supset x_0 \text{, $\cl_0$-independent over $\scl(X)$,}\\ &\text{such that } \delta(x/\scl(X)) = 0,
\end{align*}
where $\scl(X)$ denotes the \emph{strong closure} with respect to the predimension function $\delta$, i.e. the smallest strong $\spank$-closed set containing $X$. Later on it will be important to know the following: first, the strong closure of any set exists; second, if a set has finite linear dimension then its strong closure also has finite linear dimension (and in particular is countable); the strong closure of any set is the union of the strong closures of its finite dimensional subsets, hence the strong closure of a countable set is always countable (see \cite[Lemma 2.9]{CayGreenTh}).

Let us assume the Schanuel Conjecture for the rest of this subsubsection. Thus, the predimension function $\delta_{\exp}$ is non-negative and, we have an associated dimension function $\dd_{\exp}$ on $\C_{\exp}$, and a corresponding pregeometry $\cl_{\exp}$

\begin{Def} \label{pred2dim}
For any $\A \in \CC_0$, the \emph{dimension function $\dd$ associated to $\delta$} is defined for all finite $X \subset A$ by the formula
\[
\dd(X) = \min \{ \delta(X') : X \subset X' \subsetfin A\}.
\]
\end{Def}

\begin{Rem} \label{pred2cl}
The function $\dd$ has the following properties: 
\begin{itemize}
\item $\dd(\emptyset) = 0$.
\item For all $X,Y \subsetfin A$, if $X \subset Y$ then $\dd(X) \leq \dd(Y)$.
\item For all $X,Y,Z \subsetfin A$, if $\dd(XY) = \dd(Y)$ then $\dd(XYZ) = \dd(YZ)$.
\footnote{Equivalently, for all $Y,Z \subsetfin A$ and all $x \in A$, if $\dd(xY) = \dd(Y)$ then $\dd(xYZ) = \dd(YZ)$.}
\end{itemize}
It follows that associated to $\dd$ we have a closure operator with finite character $\cl_{\dd}$ on $A$ which restricts to a pregeometry on the set $A_1 := \{ x \in A : \dd(x) \leq 1 \}$ with dimension function $\dd$. Indeed, the operator $\cl_{\dd}$ is given by: for $X \subsetfin A$ and $x_0 \in A$, 
\begin{align*}
     &x_0 \in \cl_{\dd}(X)\\
\iff &\dd(x_0/X) = 0\\
\iff &\dd(x_0 X) = \dd(X)\\
\iff &\text{ There exists a tuple } x \supset x_0 \text{ such that } \delta(x/\scl(X)) = 0\\
\iff &x_0 \in \scl(X) \text{ or } \text{there exists a tuple } x \supset x_0 \text{, $\cl_0$-independent over $\scl(X)$,}\\ &\text{such that } \delta(x/\scl(X)) = 0.   
\end{align*}
\end{Rem}

The following two definitions will be needed to understand $\cl_{\exp}$ in the proof of Lemma~\ref{CCPexp}. 

\begin{Def} \label{def-ex-rotund}
Let $n \geq 1$. A subvariety $W$ of $\C^n \times (\C^*)^n$ is said to be \emph{$\ex$-rotund} if for every $k \times n$-matrix $M$ with entries in $\Z$ of rank $k$, $\dim W' \geq k$, where $W'$ is the image of $W$ under the map from $\C^n \times (\C^*)^n$ to $\C^k \times (\C^*)^k$ given by $(x,y) \mapsto (M \cdot x, y^M)$
\end{Def}

\begin{Def} \label{def-ex-genreal}
Let $W \subset \C^n \times (\C^*)^n$ be an $\ex$-rotund variety and let $B \subset \C$ be such that $W$ is defined over $B \cup \exp(B)$. Let us say that $a \in \C^n$ is a \emph{generic realisation} of $W$ over $B$, if $(a,\exp(a))$ is a generic point of $W$ over $B \cup \exp(B)$.
\end{Def}

Also, consider the following definition:

\begin{Def}
A pregeometry $\cl$ on a set $A$ is said to have the \emph{Countable Closure Property (CCP)} if for every finite subset $X$ of $A$, the set $\cl(X)$ is countable.
\end{Def}

Let us remark the simple fact that if a pregeometry $\cl$ has the CCP, then the \emph{localisation} $\cl_D$ of $\cl$ over a countable set $D$, i.e. the pregeometry defined by the formula $\cl_D(X) = \cl(D \cup X)$, also has the CCP.

We can now give the statement and proof of Lemma~5.12 from \cite{ZPseudoExp}.

\begin{Lemm} \label{CCPexp}
Assume SC holds. Then the pregeometry $\cl_{\exp}$ has the CCP.
\end{Lemm}
\begin{proof}
Let $B$ be a finite subset of $\C$. We shall prove that $\cl_{\exp}(B)$ is countable. By passing to its strong closure, we may assume that $B$ is strong with respect to $\delta_{\exp}$. Note that for any element $x_0 \in \C$, $x_0$ is in $\cl_{\exp}(B)$ if and only if $x_0 \in \spanQ(B)$ or there exists $x \supset x_0$, $\Q$-linearly independent over $B$, such that $\delta_{\exp}(x/B) = 0$. 

It is clear that $\spanQ(B)$ is countable, it therefore suffices to show that the set
\[
\{ x \subset \C : x \text{ is $\Q$-linearly independent over $B$ and } \delta_{\exp}(x/B) = 0\}
\] 
is also countable. 

Suppose $x$ is $\Q$-linearly independent over $B$ and let $W$ be the algebraic locus of $(x,\exp(x))$ over $B \cup \exp(B)$. Since $B$ is strong, the variety $W$ is $\ex$-rotund and, clearly, $x$ is a generic realisation of $W$ over $B$. Also, note that $\delta_{\exp}(x/B) = 0$ if and only if $\dim W = n$. Thus, it is sufficient to prove that for every $n$ and for every $\ex$-rotund variety $W \subset \C^{2n}$ defined over $B \cup \exp(B)$ of dimension $n$, the set of generic realisations of $W$ over $B$ is countable (clearly, there are only countably many such varieties $W$.) This is done below.

Let $W \subset \C^{2n}$ be an $\ex$-rotund variety defined over $B \cup \exp(B)$ of dimension $n$.

The proof of the following claim completes the proof of the lemma. 

\noindent\textbf{Claim:} Consider the (analytic) set
\[
\SW = \{ x \in \C^n : (x,\exp x) \in W \}.
\]
There is an analytic set $\SW_0$ of dimension zero contained in $\SW$ such that every generic realisation of $W$ over $B$ either is in $\SW_0$ or is an isolated point of $\SW$.\footnote{An \emph{analytic subset} of a domain $U$ in  $\C^n$ is a set that locally, around every point in $U$, is defined as the zero set of some complex analytic functions. We call analytic subsets of $\C^n$ simply \emph{analytic sets}. For precise definitions see \cite[Section 2.1]{Chirka})}

Indeed, the claim implies that the set of generic realisations of $W$ over $B$ is countable: Since $\SW_0$ is an analytic set of dimension zero, it consists of isolated points, it is therefore discrete and hence countable (for every discrete subset of Euclidean space is countable). Also, the set of isolated points of $\SW$ is clearly discrete and hence countable.

\noindent\textbf{Proof of Claim:}
Being analytic, the set $\SW$ can be written as a union $\bigcup_{0 \leq i \leq d} \SW_i$ where, for each $i$, the set $\SW_i$ is a complex manifold of dimension $d$ (possibly empty) and the union $\bigcup_{0 \leq j \leq i} \SW_j$ is an analytic set (\cite[Section 5.5]{Chirka}). In particular, the set $\SW_0$ is an analytic set of dimension 0.

Let us now show that any generic realisation of $W$ over $B$ in $\SW \setminus \SW_0$ is an isolated point of $\SW$.

Suppose not. Then there exists a generic realisation $a$ of $W$ over $B$ in $\SW \setminus \SW_0$ that is not an isolated point of the analytic set $\SW$. 

Since $a$ is in some $\SW_i$ with $i > 0$, there exists an analytic isomorphism $x:t \mapsto x(t)$ from an open disc $D$ around $0$ in $\C$ onto a subset of $\SW$ mapping $0$ to $a$. 

Set $y(t):=\exp(x(t))$. Then for every $t \in D$, $(x(t),y(t))$ is in $W$.

We can consider (the germ of) each coordinate function of $x$ and $y$ as an element of the differential ring $\mathcal{R}$ of germs near $0$ of functions which are analytic on a neighbourhood of $0$.
\footnote{The equivalence relation defining the germs is given by: $f \sim g$, if $f$ and $g$ coincide on a punctured neighbourhood of 0.}
Note that the ring of constants of $\mathcal{R}$ is (isomorphic to) $\C$. 
Using the fact that the zero set of an analytic function in one variable consists of isolated points, it is easy to see that $\mathcal{R}$ is an integral domain. Thus, $\mathcal{R}$ embeds into its field of fractions, $\mathcal{F}$. The derivation on $\mathcal{R}$ extends to a derivation on $\mathcal{F}$ (by the usual differentiation rule) with field of constants $\mathcal{C} \supset \C$.
  
Since $(x,y) \in W(\mathcal F)$, we get that
\begin{equation*}
\trd(x, y/B \cup \exp(B)) \leq n.
\end{equation*}
In fact, since $x(0)=a$ is a generic realisation over $B$, $\trd(x(0),y(0)/B \cup \exp(B)) = n$, and hence
\[
\trd(x, y/B \cup \exp(B)) = n.
\]

Let $k \in \{0,\dots,n\}$ be the number of independent $\Q$-linear dependences among $Dx_1,\dots,Dx_n$, i.e. $k = n - \ld_{\Q}(Dx_1,\dots,Dx_n)$. After a $\Q$-linear change of coordinates we can assume that $Dx_1,\dots,Dx_k$ are all identically zero and $Dx_{k+1},\dots,Dx_n$ are $\Q$-linearly independent. Thus, $x_1,\dots,x_k$ are all constant, with values $a_1,\dots,a_k$ respectively. Since $W$ is $\ex$-rotund, we have
\begin{equation*}
\trd(a_1,\dots,a_k,\exp(a_1),\dots,\exp(a_k)/B \cup \exp(B)) \geq k.
\end{equation*}

Hence
\begin{align*}
&\trd(x_{k+1},\dots,x_n,y_{k+1},\dots,y_n/ \mathcal{C})\\
\leq &\trd(x_{k+1},\dots,x_n,y_{k+1},\dots,y_n/ \C)\\
\leq &\trd(x_{k+1},\dots,x_n,y_{k+1},\dots,y_n/B \cup \{a_1,\dots,a_k\} \cup \exp(B \cup \{a_1,\dots,a_k\}))\\
= &\trd(x_{k+1},\dots,x_n,y_1,\dots,y_n/B \cup \exp(B))\\
&-\trd(x_1,\dots,x_k,y_1,\dots,y_k/B \cup \exp(B))\\
\leq &n-k.
\end{align*}

Ax's Theorem (\cite[Statement (SD)]{Ax71}) then implies that $Dx_{k+1},\dots,Dx_n$ must be $\Q$-linearly dependent. This gives a contradiction.
\end{proof}


\begin{Rem} \label{rem-kirby}
The only use of the Schanuel Conjecture in Lemma~\ref{CCPexp} is in the assertion that $\cl_{\exp}$ is a pregeometry. By results of Kirby (Theorem 1.1 and Theorem 1.2 in \cite{KirbyExpAlg}), without assuming the Schanuel conjecture, there is a countable strong subset $D$ of $\C$ with respect to the predimension function $\delta_{\exp}$. Thus, unconditionally, for such $D$, the localisation $(\cl_{\exp})_D$ of $\cl_{\exp}$ is a pregeometry and, by precisely the same argument as in the proof of Lemma~\ref{CCPexp}, has the CCP.
\end{Rem}

\subsubsection{Raising to powers}

Let $K$ be a subfield of $\C$. The structure $\C_K$ of \emph{raising to powers in $K$} is the following two-sorted structure:
\[
(\C,+,(\lambda\cdot)_{\lambda \in K}) \xrightarrow{\exp} (\C,+,\cdot),
\]
where the structure on the first-sort is the natural $K$-vector space structure, the structure on the second sort is the usual field structure and $\exp$ is the complex exponential function.

A model-theoretic study of the above structures, in analogy with the case of $\C_{\exp}$, has been done by Zilber in \cite{ZPowers}, with additions in \cite{ZPowers2} and \cite{ZCIT}. As in the previous subsection, we are interested in a CCP result and give only a brief account of the necessary material.

Consider the predimension function $\delta_K$ defined on tuples $x \subset \C$ by
\[
\delta_K(x) := \ld_K(x) + \trd(\exp(x)) - \ld_{\Q}(x).
\]

Assume $K$ has finite transcendence degree. Then, the Schanuel Conjecture implies that $\delta_K(x) \geq - \trd(K)$ for all $x$. Indeed, $\ld_K (x) \geq \trd (x/K) \geq \trd(x) - \trd(K)$, therefore:
\begin{align*}
\delta_K(x) &= \ld_K(x) + \trd(\exp(x)) - \ld_{\Q}(x)\\
            &\geq \trd(x) -\trd(K) + \trd(\exp(x)) - \ld_{\Q}(x)\\
            &\geq -\trd(K)
\end{align*}
where the last inequality follows from the Schanuel conjecture.  

Thus, SC implies the following conjecture:

\begin{Conj}[Schanuel Conjecture for raising to powers in $K$ (SC$_K$)]
Let $K$ be a subfield of $\C$ of finite transcendence degree. Then, for all $x \subset \C$,  
\[
\delta_K(x) \geq - \trd(K).
\]
\end{Conj}

The following theorem shows that a stronger version of the Schanuel Conjecture for raising to powers in $K$ is satisfied in the case where $K$ is generated by powers that are \emph{exponentially algebraically independent}. This result is due to Bays, Kirby and Wilkie; in the form below, it follows easily from their Theorem 1.3 in \cite{BKW}. 

\begin{Thm}[Strong Schanuel Condition for $K$ (SC$_K^*$)] \label{sck-wilkie}
Suppose $K = \Q(\lambda)$ where $\lambda$ is an exponentially algebraically independent tuple of complex numbers. Then for all $x \subset \C$, 
\[
\delta_K(x) \geq 0.
\]
\end{Thm}

For the definition of exponential algebraic independence we refer to \cite{BKW}; for our purposes it suffices to know the following: exponential algebraic independence implies algebraic independence, and, if $\beta$ is a real number which is \emph{generic} in the o-minimal structure $\R_{\exp}$ (i.e. which is not in $\dcl_{\R_{\exp}}(\emptyset)$) then $\beta$ is exponentially transcendental (i.e. the singleton $\{\beta\}$ is exponentially algebraically independent). In particular, these two facts imply that if $\beta \in \R$ is generic in $\R_{\exp}$, then the complex number $\beta i$ is exponentially transcendental.


Assume SC$_K$. Then the values of the submodular predimension function $\delta_K$ are bounded from below in $\Z$. Therefore there exists a smallest strong set for $\delta_K$, namely the strong closure of the empty set. By localising $\delta_K$ over this set we obtain a non-negative predimension function. Let us denote by $\dd_K$ the associated dimension function and by $\cl_K$ the corresponding pregeometry (without explicit mention of the localisation).

\begin{Def}
A subset $L$ of $\C^n$ defined by an equation of the form
\[
M \cdot x = c,
\]
where $M$ is a $k \times n$-matrix with entries in $K$ and $c \in \C^n$, is said to be a \emph{$K$-affine subspace} of $\C^n$. If $C\subset \C$ contains all the coordinates of $c$, then we say that $L$ is defined over $C$. Note that if the matrix $M$ has rank $r$ over $K$, then the dimension of $L$, denoted $\dim L$, is $n-r$.
\end{Def}

In analogy with Definition~\ref{def-rotund}, in the case of green points, and Definition~\ref{def-ex-rotund}, in the case of exponentiation, we have the following definition, which will be essential in our arguments in Subsection~\ref{subsection:proof2}.

\begin{Def} \label{def-K-rotund}
A pair $(L,W)$ of a $K$-affine subspace $L$ of $\C^n$ and a subvariety $W$ of $(\C^*)^n$ is said to be \emph{$K$-rotund} if for any $k \times n$-matrix $m$ with entries in $\Z$ of rank $k$ we have
\[
\dim m \cdot L + \dim W^m \geq k.
\]
\end{Def}



Minor modifications of the proof of the CCP for $\C_{\exp}$ yield a proof of the CCP in the powers case under the assumption that the SC$_K$ holds. Thus, we have: 

\begin{Lemm} \label{CCP-K}
Assume SC$_K$. Then the pregeometry $\cl_K$ on $\C$ has the CCP.
\end{Lemm}

\begin{Rem} \label{rem-kirby-K}
Notice that if $D$ is a strong subset of $\C$ with respect to $\delta_{\exp}$ containing $K$, then $D$ is also strong with respect to $\delta_K$. This is due to the fact that for every set $D \subset \C$ containing $K$, the inequality $\delta_K(x/D) \geq \delta_{\exp}(x/D)$ holds for all $x \subset \C$. Indeed, this can be seen as follows:
\begin{align*}
\delta_K(x/D) &= \ld_K(x/D) + \trd(\exp(x)/\exp(D)) - \ld_{\Q}(x/D)\\
              &\geq \trd(x/D) + \trd(\exp(x)/\exp(D)) - \ld_{\Q}(x/D)\\
              &\geq \trd(x \exp(x)/D \exp(D)) -  \ld_{\Q}(x/D)\\
              &=\delta_{\exp}(x/D)
\end{align*}
Thus, the result of Kirby mentioned in Remark~\ref{rem-kirby}, which provides a countable strong set $D$ with respect to $\delta_{\exp}$, also gives a countable strong subset of $\C$ with respect to $\delta_K$ for any countable $K$, namely the strong closure of $K$ with respect to $(\delta_{\exp})_D$. 

Also, the proof of the CCP works for proving that for any countable strong set $D$ with respect to $\delta_K$, the localisation $(\cl_K)_D$ is a pregeometry with the CCP.
\end{Rem}



\subsection{Exponentiation and raising to powers on an elliptic curve}
\label{sec:elliptic}

\subsubsection{Basic setting and exponentiation} \label{ssec:elliptic}


Let $\Elg$ be an elliptic curve defined over a subfield $k_0$ of $\C$. Put $E := \Elg(\C) \subset \PP^2(\C)$. The variety $\Elg$ has an algebraic group structure with identity element $[0,1,0]$ 
and is defined by a homogeneous equation of the form:
\[
zy^2 = 4(x - e_1)(x - e_2)(x - e_3),
\]
where $e_1,e_2$ and $e_3$ are distinct complex numbers.

Associated to $\Elg$ there is a lattice $\Lambda = \omega_1 \Z + \omega_2 \Z$ in $\C$ 
and a corresponding Weierstrass function $\wp$, defined for $x$ in $\C \setminus \Lambda$ by
\begin{equation}
\wp(x) := \frac{1}{x^2} + \sum_{\w \in \Lambda\setminus\{0\}} (\frac{1}{(x-\w)^2} -
\frac{1}{\w^2}),
\end{equation}
so that the map $\expE: \C \to E$ given by 
\[
z \mapsto 
\begin{cases} 
[\wp(z): \wp'(z): 1], &\text{ if $z \not\in \Lambda$,}\\ 
O, & \text{ if $z \in \Lambda$,}
\end{cases}
\]
is a group homomorphism from the additive group of $\C$ onto $E$. The map $\expE$ is called the exponential map of $E$.

For all $x \in \C\setminus \Lambda$, $\wp$ satisfies the differential relation
\begin{equation}
(\wp'(x))^2 = 4(\wp(x) - e_1)(\wp(x) - e_2)(\wp(x) - e_3),
\end{equation}
and
\begin{equation} \label{e}
e_1 = \wp(\frac{\w_1}{2}), \quad e_2 = \wp(\frac{\w_2}{2}), \quad e_3 = \wp(\frac{\w_1+\w_2}{2}).
\end{equation}

As in Section~\ref{sec:axioms}, we denote by $\End(\Elg)$ the ring of regular endomorphisms of $\Elg$ and by $k_{\Elg}$ its field of fractions. Also, $E$ is an $\End(\Elg)$-module and we denote by $\ld$ the corresponding linear dimension. Here we identify $\End(\Elg)$ with the subring of $\C$ consisting of all $\alpha \in \C$ such that $\alpha \Lambda \subset \Lambda$. With this convention in place, for all $x \subset \C$ we have $\ld_{\End(\Elg)}(x/\Lambda) = \ld(\expE(x))$.
Finally, the $j$-invariant of $\Elg$ will be denoted by $j(\Elg)$.


Let us also consider the action of complex conjugation on the above setting. 
Throughout, we denote by $z^c$ the complex conjugate of a complex number $z$. 
The lattice $\Lambda^c$ obtained from $\Lambda$ by applying complex conjugation has an associated Weierstrass function $\wp^c$ satisfying the relation $\wp^c(z^c) = (\wp(z))^c$ for all $z\not\in\Lambda$. Let us denote by $\Elg^c$ the corresponding elliptic curve. By \ref{e}, the affine part of $\Elg^c$ is defined by the equation
\begin{equation}
  y^2 = 4(x - e^c_1)(x - e^c_2)(x - e^c_3).
\end{equation}
Also, since $j$ is the value of a rational function on $e_1,e_2,e_3$ (see the proof of \cite[I.4.5]{Silverman2}), $j(\Elg^c) = j(\Elg)^c$.

\subsubsection{The Elliptic Schanuel Conjecture}

The following is the Elliptic Conjecture from \cite{Bertolin}. There it is 
shown to be an instance of more general conjectures of Grothendieck and
Andr\'e. We will refer to it as the \emph{Elliptic Schanuel Conjecture (ESC)}.

Let us start by introducing some conventions from the theory of elliptic integrals. Given an element $y \in E$, an \emph{integral of the first kind} is a preimage of $y$ under the exponential map $\expE$. A \emph{period} of $E$ is an integral of the first kind of the point $O$, i.e. an element of $\Lambda$. \emph{Integrals of the second kind} are more difficult to describe and, although they appear in the statement of the ESC below, we will not need to use their definition. \emph{Quasiperiods} are integrals of the second kind of the point $O$. For complete definitions we refer to Section~I.5 in \cite{Silverman2}.

We assume that the generators $\omega_1$ and $\omega_2$ of the lattice $\Lambda$ of periods satisfy $\Im(\omega_2/\omega_1) > 0$, and let $\eta_1$ and $\eta_2$ be corresponding quasiperiods, so that the Legendre relation $\omega_2 \eta_1 - \omega_1 \eta_2 = 2\pi i$ holds (\cite[I.5.2]{Silverman2}).

In the rest of this subsection, given a tuple $y =(y_1, \dots, y_r)$ of points on the curve $E$, let us denote by $x = (x_1,\dots, x_r)$ and $z = (z_1,\dots, z_r)$ corresponding integrals of the first and the second kind, respectively.  

\begin{Conj}[Elliptic Schanuel Conjecture (ESC)] \label{ESC}
  Let $\Elg^1,\dots,\Elg^n$ be pairwise non-isogenous elliptic curves. For any
  tuples $y^\nu = (y^\nu_1,\dots,y^\nu_{r^\nu})$ of points of $\Elg^\nu$, $\nu=1,\dots,n$, we have:
  \begin{multline} \label{trdineq}
    \trd(j(\Elg^\nu), \omega^\nu_1, \omega^\nu_2, \eta^\nu_1, \eta^\nu_2, y^\nu,
    x^\nu,z^\nu)_\nu\\
    \geq 2 \sum_\nu \ld_{k_{\Elg_\nu}} (x^\nu/\Lambda^\nu) + 4 \sum_\nu (\ld_{\Q} k_{\Elg_\nu})^{-1} -n + 1
  \end{multline}
\end{Conj}

In fact, we do not need to deal directly with the quasiperiods or the
integrals of the second kind for our purposes, for we can use a
consequence of the conjecture that ignores the precise contribution of
these points to the transcendence degree on the left hand side of inequality 
(\ref{trdineq}) by using obvious upper bounds. Let us therefore show
that the above conjecture implies the following simpler statement:

\begin{Conj}[Weak Elliptic Schanuel Conjecture (wESC)] 
\label{wESC} Let $\Elg^1,\dots,\Elg^n$ be pairwise non-isogenous
  elliptic curves. For any tuples $x^\nu \in \C^{r^\nu}$, $k_{\Elg_\nu}$-linearly independent over
  $\Lambda^\nu$, $\nu=1,\dots,n$, we have:
  \begin{equation}
  \trd(j(\Elg^\nu), x^\nu,\expE^\nu(x^\nu))_\nu \geq \sum_\nu r^\nu.
  \end{equation}
\end{Conj}

\begin{proof}[Proof of ESC (\ref{ESC}) $\Rightarrow$ wESC (\ref{wESC})]
Let $\Elg^1,\dots,\Elg^n$ be pairwise non-isogenous elliptic curves. For $\nu=1,\dots,n$, let $x^\nu \in \C^{r^\nu}$ be $k_{\Elg_\nu}$-linearly independent over $\Lambda^\nu$. Set $y^\nu = \expE^\nu(x^\nu)$. Then, by \ref{ESC},
\begin{multline}
\trd(j(\Elg^\nu), \omega^\nu_1, \omega^\nu_2, \eta^\nu_1, \eta^\nu_2, y^\nu, x^\nu,z^\nu)_\nu\\
\geq 2 \sum_\nu r^\nu + 4 \sum_\nu (\ld_{\Q} k_{\Elg_\nu})^{-1} -n + 1.
\end{multline}
Without loss of generality let us assume that $\Elg^1,\dots,\Elg^l$ have no CM and $\Elg^{l+1},\dots,\Elg^n$ have CM, $0\leq l \leq n$. Then $\sum_\nu (\ld_{\Q} k_{\Elg_\nu})^{-1} = l + \frac{1}{2}(n-l)$.

Thus,
\begin{equation} \label{wESCrj} 
\trd(j(\Elg^\nu), \omega^\nu_1, \omega^\nu_2, \eta^\nu_1, \eta^\nu_2, y^\nu,
x^\nu,z^\nu)_\nu
\geq 2 \sum_\nu r^\nu + 4l + 2(n-l) -n + 1.
\end{equation}

For each $\nu$, the Legendre relation $\omega^\nu_2 \eta^\nu_1 - \omega^\nu_1 \eta^\nu_2 = 2\pi i$ holds. In particular, restricting our attention to $\Elg_1,\dots,\Elg_l$, this gives
\[
\trd(\omega^\nu_1, \omega^\nu_2, \eta^\nu_1, \eta^\nu_2/2\pi i)_{\nu=1,\dots,l} \leq 3l.
\]

In the CM case, hence for $\nu=l+1,\dots,n$, there are further algebraic dependences. Indeed, it is clear that in this case $\omega^\nu_1$ and $\omega^\nu_2$ are $\Q\alg$-linearly dependent and, in fact, by a theorem of Masser (\cite{Masser}[3.1 Theorem III]), $1, \omega^\nu_1, \eta^\nu_1, 2\pi i$ form a $\Q\alg$-linear basis of the $\Q\alg$-linear span of $1, \omega^\nu_1, \omega^\nu_2, \eta^\nu_1, \eta^\nu_2, 2\pi i$. Therefore
\[
\trd(\omega^\nu_1,\omega^\nu_2,\eta^\nu_1,\eta^\nu_2/2 \pi i)_{\nu=l+1,\dots,n} \leq n-l.
\]

Combining the last two inequalities we get
\[
\trd(\omega^\nu_1,\omega^\nu_2,\eta^\nu_1,\eta^\nu_2)_{\nu=1,\dots,n} \leq 3l + (n-l) + 1.
\]
  
Thus, inequality \ref{wESCrj} implies the following:
\[
\trd(j(\Elg^\nu), y^\nu,x^\nu,z^\nu)_\nu \geq 
\big( 2 \sum_\nu r^\nu + 4l + 2(n-l) -n + 1 \big) - \big( 3l + (n-l) + 1 \big).
\]
Therefore
\[
\trd(j(\Elg^\nu), y^\nu,x^\nu)_\nu \geq \sum_\nu r^\nu.
\]
\end{proof}


Consider the case of a single elliptic curve $\Elg$ defined over $k_0 \subset \C$. 
Let $E = \Elg(\C)$.


Let us define a predimension function $\delta_{\expE}$ on $\C$ as follows: for all $x \subset \C$, let 
\[
\delta_{\expE}(x) := \trd(j(\Elg), x,\expE(x)) - \ld_{k_{\Elg}}(x/\Lambda).
\]

The wESC is clearly equivalent to the statement that for all $x \subset \C$, $\delta_{\expE}(x) \geq 0$; which means that if the wESC holds, then the predimension function $\delta_{\expE}$ is non-negative. Thus, assuming the wESC, we have an associated dimension function $\dd_{\expE}$ and corresponding pregeometry $\cl_{\expE}$. Using the same argument as in Zilber's proof of the CCP (\ref{CCPexp}), this time applying the version of Ax's theorem for the Weierstrass $\wp$-functions from \cite{KirAx}, one can see that the pregeometry $\cl_{\expE}$ has the CCP.



\subsubsection{Raising to powers on $\Elg$}

Fix a subfield $K$ of $\C$ extending $k_{\Elg}$. 


The two-sorted structure $\Elg_K$ of \emph{raising to powers in $K$ on $\Elg$} is given by:
\[
(\C,+,(\lambda\cdot)_{\lambda \in K}) \xrightarrow{\expE} (E,(W(\C))_{W \in L_{\Elg}}).
\]
where the first sort has the natural $K$-vector field structure, the second sort has the algebraic structure on $E$, and the map $\expE$ is the exponential map of $E$.


Consider the predimension function $\delta_{E,K}$ defined on tuples $x \subset \C$ by
\[
\delta_{E,K}(x) = \ld_K(x) + \trd(j(\Elg),\wp(x)) - \ld_{k_{\Elg}}(x/\Lambda).
\]

If $K$ has finite transcendence degree, then the Weak Elliptic Schanuel Conjecture (\ref{wESC}) implies that the inequality $\delta_{E,K}(x) \geq - \trd(K)$ holds for all $x \subset \C$. 

Let us state this consequence of the wESC for a single elliptic curve $\Elg$ as an independent conjecture.

\begin{Conj}[Weak ESC for raising to powers in $K$ on $E$ (wESC$_K$)] \label{wESCK}
Let $\Elg$ be an elliptic curve. Let $K$ is a subfield of $\C$ extending $k_{\Elg}$ of finite transcendence degree. Then for all $x \subset \C$, 
\[
\delta_{E,K}(x) \geq - \trd(K).
\]
\end{Conj}

Assume wESC$_{K}$ holds. Then, by localising over the strong closure of the empty set, we obtain a non-negative predimension function from $\delta_{E,K}$, for which we have an associated dimension function, which we shall denote $\dd_{E,K}$, and pregeometry, which will be denoted by $\cl_{E,K}$.
The same argument as in the proof of \ref{CCPexp}, using the version of Ax's theorem for Weierstrass $\wp$-functions from \cite{KirAx}, shows that for any countable $K$, $\cl_{E,K}$ has the CCP.


Let us extend Definition~\ref{def-K-rotund} from the multiplicative case to include the elliptic curve case. Since there is no space for confusion, we keep the same terminology.

\begin{Def} \label{def-K-rotund-elliptic}
A pair $(L,W)$ of a $K$-affine subspace $L$ of $\C^n$ and an algebraic subvariety $W$ of $\Elg^n$ is said to be \emph{$K$-rotund} if for any $k \times n$-matrix $m$ with entries in $\End(\Elg)$ of rank $k$ we have we have
\[
\dim m \cdot L + \dim m \cdot W \geq k.
\]
\end{Def}



\section{Models on the complex numbers: the multiplicative group case} \label{section:green-model}

In this section, we will find models for the theories of green points in the multiplicative group case. 

\subsection{The Models}

Throughout this section, let $\Alg = \G_m$ and $A = \Alg(\C) = \C^*$. Since we work in the multiplicative group, we shall use multiplicative notation. We also use the expressions \emph{multiplicatively (in)dependent} instead of $\End(\Alg)$-linearly (in)dependent.


Let $\eps\in \C\setminus(\R \cup i \R)$ and let $Q$ be a non-trivial divisible subgroup of $(\R,+)$ of finite rank. Put
\[
G = \exp(\eps \R + Q).
\]
Note that $G$ is a divisible subgroup of $\C^*$. 

We assume henceforth that $\eps$ is of the form $1 + \beta i$ with $\beta$ a non-zero real number, for we can always replace any $\eps\in \C\setminus(\R \cup i \R)$ for one of this form giving rise to the same $G$.

Consider the $L$-structure $(\C^*,G)$.
The following theorem is the main result of this section.

\begin{Thm} \label{theorem:green}
Let $\eps = 1 + \beta i$, with $\beta$ a non-zero real number, and let $Q$ be a non-trivial divisible subgroup of $(\R,+)$ of finite rank. Let
\[
G = \exp(\eps \R + Q).
\]
Assume SC$_K$ for $K = \Q(\beta i)$. Then:
\begin{enumerate}
\item For every tuple $c \subset \C^*$, there exists a tuple $c' \subset \C^*$ extending $c$, such that $c'$ is strong with respect to the predimension function $(\delta_G)_{c'}$. If $c \subset G$, then we can find such a $c'$ also contained in $G$.
\item The structure $(\C^*,G)$ has the EC-property. Therefore, for every tuple $c \subset G$, strong with respect to $(\delta_G)_c$, the structure $(\C^*,G)_{X_0}$ is a model of the theory $T_{X_0}$, where $X_0 = \spank(c)$ with the structure induced from $(\C^*,G)$.
\end{enumerate} 
\end{Thm}

The above theorem follows immediately from Propositions \ref{prop:green1} and \ref{prop:green2} below. Subsections \ref{subsection:proof1} and \ref{subsection:proof2} are devoted to the corresponding proofs.


\subsection{The Predimension Inequality} \label{subsection:proof1}

In this subsection we prove the first part of Theorem~\ref{theorem:green}. The proof here improves upon the corresponding one in \cite{ZBicol}.

\begin{Lemm} \label{lemma:predim-ineq} 
Let $K = \Q(\beta i)$ and assume SC$_K$.

Then for all $y\in (\C^*)^n$, we have $\delta_G(y) \geq - 3 \ld_{\Q} Q - \trd(K) $.
\end{Lemm}
\begin{proof}
We may assume $y\in G^n$ and is multiplicatively independent. Let $x \in \C^n$ be such that $\exp(x) = y$ with $x = \eps t + q$, $t\in \R^n$, $q\in Q^n$. Note that $x$ is $\Q$-linearly independent over the kernel of $\exp$.

Since complex conjugation is a field automorphism of $\C$, we have
\begin{equation} \label{eq:pd1}
2 \trd(y) = \trd(y) + \trd(y^c).
\end{equation}
Also,
\begin{equation} \label{eq:pd2}
\trd(y) + \trd(y^c) \geq \trd(y y^c) = \trd(\exp(x)\exp(x^c)).
\end{equation}
By the SC$_K$,
\begin{equation*}
\ld_K(x x^c) + \trd(y y^c) - \ld_{\Q}(x x^c) \geq -\trd(K)
\end{equation*}
and therefore
\begin{equation} \label{eq:pd3}
\trd(y y^c) \geq \ld_{\Q}(x x^c) - \ld_K(x x^c) -\trd(K)
\end{equation}
Combining \ref{eq:pd1},\ref{eq:pd2} and \ref{eq:pd3}, we obtain 
\[
2 \trd(y) \geq \ld_{\Q}(x x^c) - \ld_K(x x^c) -\trd(K).
\] 
Thus, in order to prove the lemma, it is sufficient to show that the difference $\ld_{\Q}(x x^c) - \ld_K(x x^c)$ is always at least $n - 3 \ld_{\Q} Q$.

Let $l:=\ld_{\Q} Q$. Since $x = \eps t + q$, we have $x^c = \eps^c t + q$. Also:
\[
\frac{\eps^c}{\eps} 
= \frac{1-\beta i}{1+ \beta i} 
\in \Q(\beta i) = K.
\]
From this we obtain the following upper bound for $\ld_K(xx^c)$:
\begin{equation} \label{eq:bound1}
\ld_K(x x^c) \leq \ld_K(\eps t, q) \leq n + \ld_{\Q} Q = n + l.
\end{equation}

We now need to bound $\ld_{\Q}(x x^c)$ from below. Note that the values $\ld_{\Q}(x x^c)$ and $\ld_K(x x^c)$ do not change if we replace $x$ by any $x'$ with the same $\Q$-linear span (and $x^c$ by $x'^c$ accordingly). It follows that we can assume that for every $i \in \{l +1, \dots, n\}$, $q_i = 0$. Indeed, one can apply appropriate regular $\Q$-linear transformations to $x$ (and accordingly to $x^c$) to reduce to this case.

Since $x$ is linearly independent, in particular we have $\ld_{\Q}(x_{l +1},\dots,x_n) = n - l$, i.e.  $\ld_{\Q}(\eps t_{l +1},\dots,\eps t_n) = n - l$.  Moreover, since $\eps \not\in \R \cup i\R$, $\eps$ and $\eps^c$ are $\R$-linearly independent. Therefore
\[
\ld_{\Q}(\eps t_{l +1},\dots,\eps t_n, \eps^c t_{l +1},\dots,\eps^c t_n) = 2 (n - l).
\]
Thus,
\begin{equation}
\label{eq:bound2}
\ld_{\Q}(x x^c) \geq \ld_{\Q}(x_{l +1},\dots,x_n,x^c_{l +1},\dots,x^c_n) \geq 2n - 2l.
\end{equation}
From \ref{eq:bound1} and \ref{eq:bound2} we conclude
\[
\ld_{\Q}(x x^c) - \trd(x x^c) \geq (2n- 2l) - (n + l) = n - 3 l.
\]
\end{proof}


\begin{Prop} \label{prop:green1}
Assume SC$_K$ for $K = \Q(\beta i)$. 
Then for every tuple $c \subset \C^*$, there exists a tuple $c' \subset \C^*$, extending $c$, such that $c'$ is strong with respect to the predimension function $(\delta_G)_{c'}$. If $c \subset G$, then we can find such a $c'$ also contained in $G$.
\end{Prop}
\begin{proof}
By Lemma \ref{lemma:predim-ineq}, the set of values of $\delta_G$ on $(\C^*,G)$ is bounded from below in $\Z$. We can therefore find a tuple $c^0$ such that $\delta_G(c^0)$ is minimal. Since for every $\spank$-closed set $X$, $\delta_G(X) \geq \delta_G(X\cap G)$, we can find such $c_0$ with all its coordinates in $G$. Clearly, $c^0$ is strong for $\delta_G$, hence the localisation $(\delta_G)_{c_0}$ is a non-negative predimension function on $\C^*$.

For every $c \subset \C^*$, let $c' \subset \C^*$ be a tuple containing both $c$ and $c_0$ that generates the strong closure of $c$ with respect to $(\delta_G)_{c_0}$. The tuple $c'$ is, by definition, strong for $(\delta_G)_{c_0}$. Since $c_0 \supset c'$, it follows that $c'$ is strong for $(\delta_G)_{c'}$. It is easy to see, that if $c$ is contained in $G$, then $c'$ can be taken to be contained in $G$.
\end{proof}


\subsection{Existential Closedness} \label{subsection:proof2}

This subsection is devoted to the proof of the following proposition:

\begin{Prop} \label{prop:green2}
The structure $(\C^*,G)$ has the EC-property. Therefore, for every tuple $c \subset G$, strong with respect to $(\delta_G)_c$, the structure $(\C^*,G)_{X_0}$ is a model of the theory $T_{X_0}$, where $X_0 = \spank(c)$ with the structure induced from $(\C^*,G)$.
\end{Prop}

For the rest of Subsection~\ref{subsection:proof2}, let us fix an even number $n\geq 1$ and a rotund variety $V\subset(\C^*)^n$ of dimension $\frac{n}{2}$ defined over $k_0(C)$ for some finite $C \subset \C$. We need to show that the intersection $V \cap G^n$ is Zariski dense in $V$. 

Let us define the set 
\[
\X = \{ (s,t) \in \R^{2n} : \exp(\eps t + s) \in V \}.
\]
Note that if $(s,t)$ is in $\X \cap (Q^n \times \R^n)$, then the corresponding point $y : = \exp(\eps t + s)$ is in $V \cap G^n$. Thus, in order to find points in the intersection $V \cap G^n$, we shall look for points $(s,t)$ in $\X$ with $s \in Q^n$.

Our strategy for this is to find an implicit function for $\X$ defined on an open set $S \subset \R^n$, assigning to every $s \in S$ a point $t(s) \in \R^n$ such that $(s,t(s)) \in \X$. Since $Q^n$ is dense in $\R^n$, the intersection $S \cap Q^n$ is non-empty, and therefore we can find points $(s,t(s))$ in $\X$ with $s \in Q^n$. 

Let $\RR$ be the expansion of the real ordered field by the restrictions of the real exponential function and the sine function to all bounded intervals with rational end-points, and by constants for the real and imaginary parts of the elements of $\Q(C)$.

Since $\RR$ is an expansion by constants of a reduct of $\R_{an}$, $\RR$ is o-minimal. Note that the set $\X$ is \emph{locally definable} in $\RR$, i.e. its intersection with any bounded box with rational endpoints is a definable set in $\RR$.

Let us briefly introduce some conventions and basic facts from dimension theory in o-minimal structures. Firstly, if $\RR$ is o-minimal, the definable closure $\dcl_{\RR}$ coincides with the algebraic closure $\acl_{\RR}$ and it is a pregeometry (that $\dcl_{\RR}$ satisfies the exchange axiom follows from the Monotonicity Theorem (\cite[3.1.2]{vdD})).
The dimension function associated to the pregeometry $\dcl_{\RR}$ will be denoted by $\dim_{\RR}$.

For expansions of the reals, we have the following key fact: Suppose $\RR$ is an expansion of the real ordered field in a countable language. Then for any $X \subset \R^n$ that is locally definable in $\RR$ over a countable set $A$ we have
\[
\max_{x\in X} \dim_{\RR} (x/A) = \dim_\R X,
\]
where $\dim_\R X$ is the \emph{topological dimension} of $X$, i.e. the maximum $ k\leq n$ such that for some coordinate projection $\pi$ from $\R^n$ to $\R^k$, the set $\pi(X)$ has interior.
If $X$ is a real analytic set, then $\dim_\R X$ is also its \emph{real analytic dimension}, i.e. the maximum $k$ such that for some $x \in X$ and open neighbourhood $V_x$ of $x$, $X \cap V_x$ is a real analytic submanifold of $\R^n$ of dimension $k$. (The first part of the above fact follows easily from the Baire Category Theorem, as noted in \cite[Lemma 2.17]{HruPillay}. The second part is a standard fact in real analytic geometry.)

If $X \subset R^n$ is a locally definable set in $\RR$ over a $A \subset R$, an element $b$ of $X$ is said to be \emph{generic in $X$ over $A$} if 
\[
\dim_{\RR} (b/A) = \max_{x\in X} \dim_{\RR} (x/A).
\]


Our proof of Proposition~\ref{prop:green2} relies on the following lemma, whose proof we postpone until the next subsection.

\begin{Lemm}[\textbf{Main Lemma}] \label{MainLemma}
Suppose $(s^0,t^0)$ is an $\RR$-generic point of $\X$, i.e. $\dim_\RR(s^0,t^0) = \dim_\R \X = n$. Then $\dim_\RR(s^0) = n$.
\end{Lemm}

Let us now continue with the proof of the existential closedness, using the Main Lemma.

\begin{Lemm} \label{ImplicitG}
Suppose $(s^0,t^0)$ is an $\RR$-generic point of $\X$. There is a continuous $\RR$-definable function $s \mapsto t(s)$ defined on a neighbourhood $S\subset \R^n$ of $s^0$ and taking values in $\R^n$ such that for all $s \in S$, the point $y(s) := \exp(\eps t(s) + s)$ is in $V$.
\end{Lemm}
\begin{proof}
Let $\pi: \R^{2n} \to \R^n$ be the projection onto the first $n$ coordinates. 

Let $\X^0$ be the intersection of $\X$ and a box with rational end-points containing $(s^0,t^0)$. By the Main Lemma (\ref{MainLemma}), $\dim_{\RR}(s^0) = n$. Since $\pi(\X^0)$ is definable in $\RR$, we have $\dim_{\R} \pi(\X^0) = \max_{s\in \pi(\X^0)} \dim_{\RR} (s) \geq \dim_{\RR}(s^0) = n$. Therefore $\dim_{\R} \pi(\X^0) = n$, and hence the set $\pi(\X^0)$ contains an open neighbourhood $S$ of $s^0$. 

By the definable choice property of o-minimal expansions of ordered abelian groups (\cite[6.1.2]{vdD}), there is an $\RR$-definable map $t: \pi(\X^0) \to \R^n$ such that for all $s \in \pi(\X^0)$, $(s,t(s))$ is in $\X^0$. In particular, for all $s \in S$, $(s,t(s)) \in \X$, i.e. $y(s) := \exp(\eps t(s) + s)$ is in $V$.

The o-minimality of $\RR$ also gives that the set of points where the $\RR$-definable function $t$ is discontinuous is $\RR$-definable and of dimension strictly lower than $n$ (this follows from  the $C^1$-Cell Decomposition Theorem \cite[7.3.2]{vdD}, together with the fact that the boundary of any subset of $\R^n$ has dimension strictly less than $n$ \cite[4.1.10]{vdD}). Thus, by making $S$ smaller if necessary, we may assume $t$ is continuous on $S$.
\end{proof}

\begin{proof}[Proof of Proposition~\ref{prop:green2}] 
Let $V'$ be a proper subvariety of $V$. We need to see that the intersection $(V \setminus V') \cap G^n$ is non-empty.

Extending $C$ if necessary, we may assume that $V$ and $V'$ are defined over $C$.

Take an element $y^0$ of $V \setminus V'$ with $\dim_\RR(y^0) = \dim_\R V = n$. Let $x^0 \in \C^n$ be such that $\exp(x^0) = y^0$ and let $t^0, s^0 \in \R^n$ be such that $x^0 = \eps t^0 + s^0$. 

Note that $\dim_{\RR} (s^0,t^0) = \dim_{\RR}(x^0) = \dim_{\RR}(y^0) = n$. Hence $(s^0,t^0)$ is $\RR$-generic in $\X$.

Let $S$ and the map $s \mapsto t(s)$ be as provided by Lemma~\ref{ImplicitG} for $\RR$ and $(s^0,t^0)$. Consider the map $s \mapsto y(s) := \exp(\eps t(s) + s)$ defined on $S$. This map is continuous, hence $y^{-1}(V')$ is a closed subset of $S$ not containing $s^0$. Thus, $S' = S \setminus y^{-1}(V')$ is an open neighbourhood of $s^0$.

Since $Q^n$ is dense in $\R^n$, we can take a point $q$ in $S' \cap Q^n$, and thus obtain a corresponding point $y(q)$ in $(V \setminus V') \cap G^n$.
\end{proof}

\subsection{Proof of the Main Lemma} \label{subsection:ProofMainLemma}

The image of $V$ under complex conjugation, $V^c$, plays an important role in our proof of the Main Lemma. Since complex conjugation is a field isomorphism, $V^c$ is also an irreducible algebraic variety defined over the set $C^c$. Extending $C$ if necessary, we may assume that $V^c$ is also defined over $C$.

\begin{Nota}
For a tuple $x$ of variables or complex numbers, the expression $\bar x$ will denote another tuple of variables or complex numbers, respectively, of the same length, bearing no formal relation to the former.  This notation is meant to imply that we are particularly interested in the case where $x$ is a complex number and $\bar x$ equals $x^c$.
\end{Nota}

Throughout this subsection, let $K = \Q(\beta i)$.

\begin{Def}
  For $s \in \C^n$, we define the set
  \[
  L_s = \{ (x,\bar x) \in \C^{2n}: (x + \bar x) + \beta^{-1} i (x - \bar x) = 2s \}.
  \]
\end{Def}

\begin{Rem}
Note that $\beta^{-1} i = - (\beta i)^{-1} \in K$, hence $L_s$ is a $K$-affine subspace.
\end{Rem}

\begin{Rem} \label{LRem}
Suppose $s$ is in $\R^n$. Then, for all $x \in \C^n$, the point $(x, x^c)$ belongs to $L_s$ if and only if $x = \eps t + s$ for some $t \in \R^n$. 

To see this, let $x \in \C^n$ be given, and let $t \in \C^n$ be such that $x = \eps t + s$. Then:
\begin{align*}
(x, x^c) \in L_s &\iff (x + x^c) + \beta^{-1} i (x - x^c) = 2s\\
&\iff 2\Real(x)+ \beta^{-1} i(2i\Imag(x)) = 2s\\
&\iff (\Real(\eps t) + s) - \beta^{-1} \Imag(\eps t) = s\\
&\iff \Real(\eps t) - \beta^{-1} \Imag(\eps t) = 0\\
&\iff (\Real(t) - \beta \Imag(t)) - \beta^{-1}(\Imag(t) + \beta \Real(t)) = 0\\
&\iff (\beta + \beta^{-1}) \Imag(t) = 0\\
&\iff \Imag(t) = 0\\
&\iff t \in \R^n.
\end{align*}
\end{Rem}

\begin{Lemm} \label{lemm:littlething}
Suppose $s \in \R^n$. Then for all linearly independent $(m^1,n^1),\dots,(m^k,n^k) \in \Z^{2n}$ ($m^i,n^i \in \Z^n$), we have
\[
\dim (m,n) \cdot L_{s} \geq \frac{k}{2}.
\]
\end{Lemm}
\begin{proof}
Suppose $(m^1,n^1),\dots,(m^k,n^k) \in \Z^{2n}$ are linearly independent ($m^i,n^i \in \Z^n$).

Let $D$ be any countable set over which $L_s$ is defined and let $t\in \R^n$ be such that $\ld_K(t/D) = n$. For $x = \eps t + s$ and $\bar x = x^c = \eps^c t + s$, the tuple $(x, \bar x)$ is in $L_{s}$, as \ref{LRem} shows. Then, we have:
\begin{align*}
\dim (m,n) \cdot L_{s} &\geq \ld_K( (m,n) \cdot (x, \bar x)/D)\\
&= \ld_K((m^1,n^1) \cdot (x, \bar x),\dots, (m^k,n^k) \cdot (x, \bar x)/D) \\
&= \ld_K((m'^1,n'^1) \cdot (t, \beta it),\dots, (m'^k,n'^k) \cdot (t, \beta it)/D) 
\end{align*}
where $m'^i = m^i + n^i$ and $n'^i = m^i - n^i$, for all $i=1,...,k$. Since $m^i = \frac{1}{2}(m'^i + n'^i)$ and $n^i = \frac{1}{2}(m'^i - n'^i)$, the matrix $(m',n')$ has the same rank as $(m,n)$, that is $k$. Therefore we can take a matrix $M \in \GL_k(\Z)$ and $t' = (t_{j_1},\dots,t_{j_l}, \beta it_{j_{l+1}},\dots, \beta it_{j_k})$, with $1 \leq l \leq k$, such that
\[
\ld_K((m'^1,n'^1) \cdot (t, \beta it),\dots, (m'^k,n'^k) \cdot (t, \beta it)/D) = \ld_K (M \cdot t'/D).
\]
Thus,
\[
\dim (m,n) \cdot L_{s} \geq \ld_K (M \cdot t'/D).
\] 
But note that $\ld_K (M \cdot t'/D)$ is at least $\frac{k}{2}$, for we have
\begin{align*}
\ld_K (M \cdot t'/D) &= \ld_K(t'/D)\\ &\geq \max \{\ld_K(t_{j_1},\dots,t_{j_l}/D), \ld_K(\beta it_{j_{l+1}},\dots, \beta it_{j_k}/D)\}\\ &= \max \{l,k-l\} \geq \frac{k}{2}.
\end{align*}
Therefore, $\dim (m,n) \cdot L_{s} \geq \frac{k}{2}$
\end{proof}

\begin{Lemm} \label{NormKG} 
Let $s \in \R^n$. Then the pair $(L_{s},V \times V^c)$ is $K$-rotund.
\end{Lemm}
\begin{proof}
Suppose $(m^1,n^1),\dots,(m^k,n^k) \in \Z^{2n}$ are linearly independent ($m^i,n^i \in \Z^n$).

The rotundity of $V$ implies that the variety $V \times V^c$ is also rotund. Hence $\dim(V \times V^c)^{(m,n)} \geq \frac{k}{2}$.

Also, by Lemma~\ref{lemm:littlething}, $\dim (m,n) \cdot L_{s} \geq \frac{k}{2}$. 

Therefore, we have:
\[
\dim (m,n) \cdot L_{s} + \dim(V \times V^c)^{(m,n)} \geq \frac{k}{2} +\frac{k}{2} = k.
\]
Thus, the pair $(L_{s},V \times V^c)$ is $K$-rotund.
\end{proof}

\begin{proof}[Proof of the Main Lemma] 
Consider the set
\[
L_{s^0} \cap \log(V\times V^c).
\]
It is an analytic subset of $\C^{2n}$ containing the point $(x^0,(x^0)^c)$. Since every analytic set can be written as the union of its irreducible components and this union is locally finite (\cite[Section 5.4]{Chirka}), there exist a neighbourhood $B$ of $(x^0,(x^0)^c)$, a positive integer $l$ and irreducible analytic subsets $S_1,\dots, S_l$ of $B$ containing $(x^0,(x^0)^c)$ such that
\[
L_{s^0} \cap \log(V\times V^c) \cap B = S_1 \cup \dots \cup S_l.
\]
We may assume $B$ is a box with rational end-points.

\begin{Claim}
Every $S_i$ has complex analytic dimension $0$.
\end{Claim}

Before proving the claim, let us show how the lemma follows. The claim implies that each $S_i$ is a closed discrete subset of $B$; since $B$ is bounded, each $S_i$ must then be finite. Being the union of the $S_i$, the set $L_{s^0} \cap \log(V\times V^c) \cap B$ is therefore finite, and it is clearly $\RR$-definable over $s^0$. Thus, the singleton $\{(x^0,(x^0)^c)\}$ is $\RR$-definable over $s^0$ as the intersection of $L_{s^0} \cap \log(V\times V^c) \cap B$ and a sufficiently small $\RR$-definable open box around $(x^0,(x^0)^c)$. Therefore $\dim_{\RR}(s^0) = \dim_{\RR}(x^0) = n$.

\begin{proof}[Proof of the claim]
Suppose towards a contradiction that there exists $i$ such that the set $S:=S_i$ is of positive dimension.
  
Let us show that there are uncountably many points in $S$ whose image under exponentiation is a generic point of $V \times V^c$ over $C$. 

To see this, suppose $V'$ is a proper subvariety of $V\times V^c$ over $C$. Note that $(y^0,(y^0)^c)$ is a generic point of $V \times V^c$ over $C$, for we have
\[
\trd(y^0,(y^0)^c/C) = \trd(\Real(y^0),\Imag(y^0)/C) \geq \dim_{\RR}(\Real(y^0),\Imag(y^0)) = n = \dim V \times V^c.
\]
Hence $(x^0,(x^0)^c)$ does not belong to $\log V'$, and therefore $S \cap \log V'$ is an analytic subset of $B$ properly contained in $S$.  Then, by the irreducibility of $S$, for any such $V'$, $S\cap \log V'$ is nowhere-dense in $S$. Since $S$ has positive dimension we can apply the Baire Category Theorem to conclude that there exist uncountably many $(x,\bar x)$ in $S$ that do not belong to $\log V'$ for any such $V'$, i.e. their images under exponentiation are generic points of $V \times V^c$ over $C$. 

Let $D$ be a countable strong subset of $\C$ with respect to $\delta_K$ (provided by Remark~\ref{rem-kirby-K}). Let $D'$ be the strong closure of $\log C \cup s^0$ with respect to $(\delta_K)_D$. 


For any tuple $z \subset \C^*$, if $\delta_K(z/D') \leq 0$ then all the coordinates of $z$ lie in $\cl_K(D')$. But $\cl_K(D')$ is countable, because $D'$ is countable and $(\cl_K)_D$ has the Countable Closure Property, so there can be no more than countably many tuples $z$ with $\delta_K(z/D') \leq 0$.
Thus, we can find $(x, \bar x)\in L_{s^0}$ such that $(\exp(x), \exp (\bar x))$ is a generic point of $V \times V^c$ over $C$ and $\delta_K(x, \bar x/D') > 0$. 

Then:
\begin{equation} \label{DeltaPosG} 
0 < \delta_K(x\bar x/D') \leq \dim L_{s^0} \cap N + \dim (V\times V^c) \cap \exp N - \dim N,
\end{equation}
where $N$ is the minimal $\Q$-affine subspace over $D'$ containing the point $(x, \bar x)$.

Since $\dim L_{s^0} = \dim (V\times V^c) = n$, it immediately follows from the inequality above that $N$ cannot be the whole of $\C^{2n}$, as in that case the right hand side would be $0$. Therefore $\dim N < 2n$.

Thus, there exist $k \geq 1$ and linearly independent $m^1,\dots,m^k \in \Z^{2n}$ such that $N$ is a translate of the subspace of $\C^{2n}$ defined by the system of equations $m^i \cdot (z,\bar z) = 0$ ($i=1,\dots,k$). Note that $\dim N = 2n -k$.

Note that $(V \times V^c) \cap \exp N$ is a generic fibre of the map $(\ )^m$ on $(V \times V^c)$, for it contains the generic point $(y, \bar y)$ of $V \times V^c$ over $C$. The addition formula for the dimension of fibres of algebraic varieties then gives
\[
\dim (V \times V^c)^m = \dim V \times V^c - \dim (V \times V^c) \cap \exp N.
\]
Also, by the addition formula for the dimension of $K$-affine subspaces,
\[
\dim m \cdot L_{s^0} = \dim L_{s^0} - \dim L_{s^0} \cap N.
\]

Adding up the two equations,
\[
\dim m \cdot L_{s^0} + \dim (V \times V^c)^m = 2n - (\dim L_{s^0} \cap N + \dim (V \times V^c) \cap \exp N)
\]
Using \eqref{DeltaPosG} we get
\[
\dim m \cdot L_{s^0} + \dim (V \times V^c)^m < 2n - \dim N = k.
\]
This implies that the pair $(L_{s^0},V \times V^c)$ is not $K$-rotund, contradicting \ref{NormKG}.
\end{proof}
\end{proof}

\subsection{The question of $\omega$-saturation} \label{sec:omegasat}

A natural question which we are unable to answer is whether the model $(\C^*,G)$ is $\omega$-saturated. Here we present two remarks on the issue.

First we show that, assuming the unproven \emph{CIT with parameters} (Conjecture~\ref{cit-param} below), we can prove an a priori stronger version of Proposition~\ref{prop:green2} which is implied by $\omega$-saturation. In fact, if the CIT with parameters holds, then a stronger form of the EC-property holds in all models of $T$.

The following is the statement of the CIT with parameters, it is a consequence of the \emph{Conjecture on Intersections with Tori (CIT)} (Conjecture~1 in \cite{ZCIT}). For details, see Theorem 1 in \cite{ZCIT}.

\begin{Conj}[CIT with parameters] \label{cit-param}
For every $k\geq 0$, every subvariety $W(x,y)$ of $(\C^*)^{n+k}$ defined over $\Q$ and every $c \in (\C^*)^k$, there exists a finite collection $H_1,\dots,H_s$ of cosets of proper algebraic subgroups of $(\C^*)^n$ with the following property:\\ for every coset $H$ of a proper algebraic subgroup of $(\C^*)^n$, if $S$ is an \emph{atypical component} of the intersection of $W(x,c)$ and $H$ (i.e. $\dim S > \dim W(x,c) + \dim H - n$), then for some $i \in \{1,\dots,s\}$, $S$ is contained in $H_i$. 
\end{Conj}

\begin{Prop}
Assume the CIT with parameters. Then every model of $T$ has the following \emph{strong EC-property}:
for any rotund variety $V\subset (K^*)^n$ of dimension $\frac{n}{2}$ defined over a finite set $C$, there exists a generic of $V$ over $C$ in $G^n$.
\end{Prop}
\begin{proof}
Let $(K^*,G)$ be a model of $T$. 
Let $V$ be a rotund variety of dimension $\frac{n}{2}$ defined over a finite set $C$.

It is sufficient to find a proper subvariety $V'$ of $V$ such that for any $y \in V \cap G^n$, if $y$ does not lie in $V'$ then $y$ is a generic point of $V$ over $C$. Indeed, that $(K^*,G)$ satisfies the EC-property guarantees that we can find a point $y \in (V \setminus V') \cap G^n$, and $y$ would then be a generic point of $V$ over $C$.

Without loss of generality we assume that $C$ is strong in $\A$. Then for any $y\in G^n$ $\delta_G(y/C) \geq 0$. In particular, for any $y$ in $V \cap G^n$, if $y$ is not a generic point of $V$ over $C$ then $y$ has to be multiplicatively dependent over $C$. 
Thus, it is enough to find a proper subvariety $V'$ of $V$ over $C$ such that for every $y \in G^n \cap V$, if $y$ is multiplicatively dependent over $C$ then $y$ is in $V'$.

By the CIT with parameters, there exist cosets $H_1,\dots,H_s$ of proper algebraic subgroups of $(K^*)^n$ such that any atypical irreducible component of the intersection of $V$ and a coset of a proper algebraic subgroup of $(K^*)^n$ is contained in some $H_i$. 

Let $V' = V \cap \bigcup_i H_i$. We shall now show that $V'$ has the required property. Suppose $y \in G^n \cap V$ is multiplicatively dependent over $C$, and let us see that $y$ then belongs to $V'$. Let $H$ be the smallest coset of a proper algebraic subgroup that is defined over $C$ and contains $y$. Let $c_H$ denote the codimension of $H$, then $c_H \geq 1$.

Let $Y$ be an irreducible component of $V\cap H$ containing $y$. Since $V\cap H$ is defined over $C$, $Y$ is defined over $\Q(C)\alg$. Then, by the predimension inequality over $C$, $2\dim Y - \dim H \geq 0$. Therefore we have
\[
2\dim Y \geq \dim H = n - c_H = 2\dim V - c_H,
\]
and consequently,
\[
\dim Y \geq \dim V - \frac{1}{2}c_H > \dim V - c_H.
\]
Hence $\dim Y > \dim V - c_H$, which means that $Y$ is an atypical irreducible component of the intersection $V\cap H$.

Indeed, the CIT with parameters then tells us that $y$ must belong to one of the $H_i$, and thus to $V'$.
\end{proof}


In the following proposition, we freely use notions and facts from \cite{CayGreenTh}.

\begin{Prop}
Assume the CIT with parameters. Suppose $(K^*,G)$ is a model of $T$ where $G$ has infinite dimension for the dimension function associated to $\delta$. Then $(K^*,G)$ is $\omega$-saturated.
\end{Prop}
\begin{proof}
In the light of \cite[Proposition~3.31]{CayGreenTh}, it is sufficient to show that $(\K^*,G)$ is \emph{rich}.

By the previous lemma, the CIT assumption implies that $(K^*,G)$ satisfies the the strong EC-property, which means that the richness property holds for \emph{prealgebraic minimal extensions}.

The assumption on the dimension of $G$ implies that that the richness property also holds for \emph{green generic minimal extensions}. This amounts to proving that for any finite strong subset $C$ of $K^*$, there exists $b \in K^*$ with $\dd_G(b/C) = 1$. But this is clear since $\dd_G(G)$ is infinite and $C$ is finite.

For \emph{minimal white generic extensions} we need to find, for any $C$ as before, an element $b \in K^*$ with $\dd_G(b/C) = 2$. We proceed by taking $b_1$ and $b_2$ with $\dd_G(b_1/C) = \dd_G(b_2/C b_2) = 1$ and setting $b = b_1 + b_2$.

It is sufficient to show that $\delta_G (b_1,b_2/C,b_1 + b_2)= 0$. Indeed, we then get
\[
0 \leq \dd_G(b_1,b_2/C,b_1 + b_2) \leq \delta_G (b_1,b_2/C,b_1 + b_2) = 0,
\]
so $\dd_G(b_1,b_2/C,b_1 + b_2) = 0$, and hence $\dd_G(b_1+b_2/C) = \dd_G(b_1,b_2/C) = 2$.

Now the calculation of $\delta_G (b_1,b_2/C,b_1 + b_2)$: By definition,
\[
\delta_G(b_1,b_2/C,b_1 + b_2) = 2 \trd(b_1,b_2/C,b_1 + b_2) - \md(b_1,b_2/C,b_1 + b_2).
\]
It is easy to see that
\[
\trd(b_1,b_2/C,b_1 + b_2) = 1.
\]
Also,
\[
\md(b_1,b_2/C,b_1 + b_2) = 2,
\]
because the variety defined by the equation $X+Y = b_1+b_2$ is rotund. Thus, $\delta_G (b_1,b_2/C,b_1 + b_2) = 2 (1) - 2 = 0$.
\end{proof}

Unfortunately, it is not clear that the dimension $\dd_G(G)$ is infinite in our model $(\C^*,G)$. Note that this would immediately follow if one could show that the corresponding pregeometry on the uncountable set $G$ has the CCP.

\subsection{Emerald points} \label{sec-eme}


In \cite{ZNCG}, a connection is established between the construction of noncommutative tori, which are basic examples of non-commutative spaces, and the model theory of the expansions of the complex field by a multiplicative subgroup of the form
\[
H = \exp(\eps \R + q \Z),
\]
where $\eps = 1 + i \beta$ and $\beta$ and $q$ are non-zero real numbers such that $\beta q$ and $\pi$ are $\Q$-linearly independent. 

In order to prove that such structures are superstable, in \cite[Section 5]{CayGreenTh}, a variant of the theories of green points was considered in which the distinguished subgroup is not divisible, but elementarily equivalent to the additive group of the integers instead. The modified theories were named \emph{theories of emerald points} and shown to be superstable. It remained to show that the above structures are in fact models of the constructed theories. 

In this subsection we remark that the arguments in this section for the green case also yield the fact that the above structures are indeed models of the theories of emerald points (after adding constants for the elements of a strong set), provided the Schanuel Conjecture for raising to powers in $K=\Q(\beta i)$ holds.

% ...

\begin{Thm} \label{theorem:emerald2}
Let $\beta$ and $q$ non-zero real numbers such that $\beta q$ and $\pi$ are $\Q$-linearly independent. Let $\eps = 1 + \beta i$ and 
\[
H = \exp(\eps \R + q\Z).
\]
Assume SC$_K$ holds for $K = \Q(\beta i)$. Then:
\begin{enumerate}
\item For every tuple $c \subset \C^*$, there exists a tuple $c' \subset \C^*$ extending $c$, such that $c'$ is strong with respect to the predimension function $(\delta_H)_{c'}$.
\item The structure $(\C^*,H)$ has the EC-property. Therefore, for every tuple $c \subset \C^*$, strong with respect to $(\delta_H)_c$, the structure $(\C^*,H)_{X_0}$ is a model of the theory $T = T_{X_0}$ from \cite[Section 5]{CayGreenTh}, where $X_0 = \spank(c)$ with the structure induced from $(\C^*,H)$.
\end{enumerate} 
\end{Thm}

The first part of the theorem follows directly from the analogous statement in the green case, by Remark~5.2 of \cite{CayGreenTh}. For the second part of the theorem, the proof of the analogous statement in the green case applies, simply using the density of $q \Z + \frac{2\pi}{\beta} \Z$ in $\R$, instead of that of the subgroup $Q$, at the very end of the proof.

\section{Models on the complex numbers: the elliptic curve case}
\label{section:elliptic-model}

In this section we find models for the theories of green points in the case of an elliptic curve without complex multiplication and whose lattice of periods is invariant under complex conjugation, under the assumption that the Weak Schanuel Conjecture for raising to powers on the elliptic curve holds. 


\subsection{The Models}

Let us fix an elliptic curve $\Elg$ without complex multiplication. Let $\E = \Elg(\C)$. We use the conventions introduced in Subsection~\ref{sec:elliptic}.

Let $\eps\in \C^*$ be such that $\eps \R \cap \Lambda = \{ 0 \}$. Put $G = \expE(\eps \R)$. 

\begin{Rem} \label{remalpha}
Note that $G$ is a divisible subgroup of $E$. 
%
%
Since $\Elg$ has no CM, $k_{\Alg} = \Q$ and, for any $y \subset E$, $\spank(y)$ is the divisible hull of the subgroup generated by $y$.
Also, $G$ is dense in $E$ in the Euclidean topology. To see this notice the following:
\[
G = \expE(\eps \R) = \expE(\eps \R + \Lambda) = \expE(\Gamma + \Z + \alpha \Z),
\]
for $\alpha = \Real(\tau) - \frac{\Real(\eps)}{\Imag(\eps)}\Imag(\tau) \in \R$. Since $\Gamma \cap \Lambda = \{0\}$, $\alpha$ is irrational. It follows that $\Z + \alpha \Z$ is dense in $\R$. Therefore the set $G = \expE(\eps + \Z + \alpha \Z)$ is dense in $\E$ in the Euclidean topology.
\end{Rem}

We now state the main theorem of this section.

\begin{Thm} \label{theorem:elliptic}
Let $\Elg$ be an elliptic curve without complex multiplication and let $\E = \Elg(\C)$. Assume the corresponding lattice  $\Lambda$ has the form $\Z + \tau\Z$ and $\Lambda = \Lambda^c$. 

Let $\eps = 1+ \beta i$, with $\beta$ a non-zero real, be such that $\eps \R \cap \Lambda = \{ 0 \}$. Put $G = \expE(\eps \R)$. 

Let $K = \Q(\beta i)$ and assume the Weak Elliptic Schanuel Conjecture for raising to powers in $K$ (wESC$_K$) holds for $\Elg$.

Then:
\begin{enumerate}
\item For every tuple $c \subset \E$, there exists a tuple $c' \subset \E$ extending $c$, such that $c'$ is strong with respect to the predimension function $(\delta_G)_{c'}$. If $c \subset G$, then we can find such a $c'$ also contained in $G$.
\item The structure $(\E,G)$ has the EC-property. Therefore, for every tuple $c \subset G$, strong with respect to $(\delta_G)_c$, the structure $(\E,G)_{X_0}$ is a model of the theory $T$, where $X_0 = \spank(c)$ with the structure induced from $(\E,G)$.
\end{enumerate} 
\end{Thm}


Let us make some remarks about the hypotheses of the theorem. 
Firstly, assuming that the lattice $\Lambda$ has generators $\omega_1 = 1$ and $\omega_2 = \tau$ is not truly restrictive, for this can always be achieved by passing to an isomorphic elliptic curve.
Secondly, the assumptions of $\Elg$ having no CM and $\Lambda$ being invariant under complex conjugation are real restrictions on the generality of the result. The first assumption is essential, since we do not have an appropriate $\End(\Alg)$-submodule of $\E$ that serves as analogue of the subgroup $G$ defined above in the CM case. The second is necessary in our arguments for proving both the predimension inequality and the existential closedness for the structure $(\E,G)$. Let us remark that the two conditions hold for any non-CM elliptic curve defined over $\R$. 
Also note that, by the remarks at the end of Subsubsection \ref{ssec:elliptic}, the assumption that $\Lambda = \Lambda^c$ implies that $\Elg = \Elg^c$ and $j(\Elg)^c = j(\Elg)$.
Finally, let us recall that our assumption that the wESC$_K$ (\ref{wESCK}) holds for the single elliptic curve $\Elg$ means the following: for any tuple $x$ of complex numbers,
\begin{equation*}
\ld_K(x) + \trd(j(\Elg),\wp(x)) - \ld_{\Q}(x/\Lambda) \geq -\trd(K).
\end{equation*}


For the rest of this section, we work under the hypotheses of the theorem, that is: $\Elg$ has no CM, $\Lambda = \Z + \tau\Z$, and $\Lambda = \Lambda^c$.
Also, $\eps = 1+ \beta i$ with $\beta$ a non-zero real and $\eps \R \cap \Lambda = \{ 0 \}$, and $G = \expE(\eps \R)$. We set $K = \Q(\beta i)$ and assume the Weak Elliptic Schanuel Conjecture for raising to powers in $K$ holds for $\Elg$.

As in the previous section, we divide the proof of the theorem into the proofs of two propositions, Propositions \ref{prop:elliptic1} and \ref{prop:elliptic2}.

\subsection{The Predimension Inequality} \label{subsection:elliptic1}

In this subsection we prove the first part of Theorem~\ref{theorem:elliptic}.

\begin{Lemm} \label{ineqE}
For any tuple $y \subset \E$, $\delta_G(y) \geq - 4 -\trd(K) -2 \trd(k_0)$.
\end{Lemm}
\begin{proof}
It is sufficient to show that for any $n$ and any $y \in G^n$ with
$\ld_{\Q}(y) = n$, we have 
\[
2 \trd_{k_0}(y) \geq n - 4 -\trd(K) - 2 \trd(k_0).
\]

Fix such $n$ and $y$. Let $x \in (\eps \R)^n$ be such that $\expE(x) = y$. Notice that $x$ is $\Q$-linearly independent over $\Lambda$.

Note the following
\begin{align*}
2 \trd(j(\Elg),\wp(x)) &\geq \trd(j(\Elg),\wp(x),(\wp(x))^c)\\
&= \trd(j(\Elg),\wp(x), \wp^c(x^c))\\
&= \trd(j(\Elg),\wp(x x^c)).
\end{align*}

By the wESC$_K$,
\[
\ld_K(x x^c) + \trd(j(\Elg), \wp(x), \wp(x^c)) - \ld_{\Q}(x x^c/\Lambda) \geq -\trd(K).
\]

Combining the above inequalities we obtain,
\[
2 \trd(j(\Elg),\wp(x)) \geq \ld_{\Q}(x x^c/\Lambda) - \ld_K(x x^c) -\trd(K).
\]

Now, on the one hand, since $\eps$ is not in $\R \cup i\R$, we know $\eps$ and
$\eps^c$ are $\R$-linearly independent and hence $\ld_{\Q}(x x^c) = \ld_{\Q}(x) +
\ld_{\Q}(x^c) = 2n$. Therefore
\[
\ld_{\Q}(x x^c/\Lambda) \geq \ld_{\Q}(x x^c) - \ld_{\Q}(\Lambda) = 2n -2.
\]

On the other hand, since $x^c = \frac{\eps^c}{\eps} x$ and $\frac{\eps^c}{\eps} \in K$, we have
\[
\ld_K(x x^c) \leq \ld_K(x) \leq n.
\]

Thus,
\[
2 \trd(j(\Elg),\wp(x)) \geq (2n - 2) - n - \trd(K) = n -2 -\trd(K).
\]

Hence, using the additivity properties of the transcendence degree, we see that 
\begin{align*}
2 \trd(\wp(x)) &= 2 \trd(j(\Elg),\wp(x)) - 2 \trd(j(\Elg)/\wp(x))\\
               &\geq 2 \trd(j(\Elg),\wp(x)) - 2\\ 
               &\geq  n-4-\trd(K)
\end{align*}
and, similarly,
\begin{align*}
2 \trd_{k_0}(\wp(x)) &= 2 \trd(\wp(x)) - 2 \trd(k_0/\wp(x))\\ 
                    &\geq n - 4 - \trd(K) - 2 \trd(k_0).
\end{align*}
Because $\wp(x)$ and $y = \expE(x)$ are interalgebraic over $k_0$, we obtain the inequality $2 \trd_{k_0}(y) \geq n - 4 -\trd(K) -2 \trd(k_0)$. 
\end{proof}

By the same argument as in Section~\ref{section:green-model}, one derives the following proposition.

\begin{Prop} \label{prop:elliptic1}
For every tuple $c \subset \E$, there exists a tuple $c' \subset \E$, extending $c$, such that $c'$ is strong with respect to the predimension function $(\delta_G)_{c'}$. If $c \subset G$, then we can find such a $c'$ also contained in $G$.
\end{Prop}

\subsection{Existential Closedness} \label{subsection:elliptic2}

The following proposition completes the proof of Theorem~\ref{theorem:elliptic}.

\begin{Prop} \label{prop:elliptic2}
The structure $(A,G)$ has the EC-property. Therefore, for every tuple $c \subset G$, strong with respect to $(\delta_G)_c$, the structure $(A,G)_{X_0}$ is a model of the theory $T$, where $X_0 = \spank(c)$ with the structure induced from $(A,G)$.
\end{Prop}

The proof of the above proposition is the same as in Section~\ref{section:green-model}, with only very small differences. In order to be explicit about the differences, we review the different steps of the proof.

For the rest of Subsection~\ref{subsection:elliptic2}, let us fix an even number $n\geq 1$ and a rotund variety $V\subset(\C^*)^n$ of dimension $\frac{n}{2}$ defined over $k_0(C)$ for some finite subset $C$ of $\E$. We need to show that the intersection $V \cap G^n$ is Zariski dense in $V$.


Let us define the set 
\[
\X = \{ (s,t) \in \R^{2n} : \expE(\eps t + s) \in V \}.
\]
Note that if $(s,t)$ is in $\X \cap ((\Z + \alpha \Z)^n \times \R^n)$, where $\alpha = \Real(\tau) - \frac{\Real(\eps)}{\Imag(\eps)}\Imag(\tau)$, then the corresponding point $y : = \expE(\eps t + s)$ is in $V \cap G^n$  (see \ref{remalpha}). Thus, in order to find points in the intersection $V \cap G^n$, we shall look for points $(s,t)$ in $\X$ with $s \in (\Z + \alpha \Z)^n$.

As in the previous section, our strategy is to find an implicit function for $\X$ defined on an open set $S \subset \R^n$, assigning to every $s \in S$ a point $t(s) \in \R^n$ such that $(s,t(s)) \in \X$. Since $(\Z + \alpha \Z)^n$ is dense in $\R^n$, the intersection $S \cap (\Z + \alpha \Z)^n$ is non-empty, and therefore we can find points $(s,t(s))$ in $\X$ with $s \in (\Z + \alpha \Z)^n$.

Let $\RR$ be an o-minimal expansion of the real ordered field in a countable
language in which the function $\wp$ is locally definable (and therefore also the set $\X$) and having constants for the real and imaginary parts of each element of $k_0(C)$. The existence of such a structure $\RR$, as a reduct of $\R_{an}$, follows from the fact that the \emph{addition formula},
\[
\wp(z_1 + z_2) = - \wp(z_1) - \wp(z_2) + \frac{1}{4}(\frac{\wp'(z_1) -
\wp'(z_2)}{\wp(z_1) - \wp(z_2)})^2,
\]
allows to locally define $\wp$ in terms of its restriction to a closed
parallelogram contained in the interior of the fundamental parallelogram of
vertices $0, 1, \tau, 1+\tau$ (e.g. the one with vertices $\frac{1+\tau}{8}$, $\frac{3+\tau}{8}$,
$\frac{1+3\tau}{8}$, $\frac{3+3\tau}{8}$), around which it is analytic
(\cite{MacEll}). Indeed, this corresponds to the fact that $\expE$ is a homomorphism and its values can therefore be calculated from those of any restriction to an open subset of the fundamental parallelogram.


The proof of Proposition~\ref{prop:elliptic2} relies on the following main lemma:

\begin{Lemm}[\textbf{Main Lemma}] \label{MainLemmaE}
Suppose $(s^0,t^0)$ is an $\RR$-generic point of $\X$, i.e. $\dim_\RR(s^0,t^0) = \dim_\R \X = n$. Then $\dim_\RR(s^0) = n$.
\end{Lemm}


To prove the Main Lemma, define the following set:

\begin{Def}
  For $s \in \C^n$ we define the set
  \[
  L_s = \{ (x,\bar x) \in \C^{2n}: (x + \bar x) + \beta^{-1} i (x - \bar x) = 2s \}.
  \]
\end{Def}


With the same proof as in Section~\ref{section:green-model}, we have:

\begin{Lemm} \label{lemm:littlethingE}
Suppose $s \in \R^n$. Then for all linearly independent $(m^1,n^1),\dots,(m^k,n^k) \in \Z^{2n}$ ($m^i,n^i \in \Z^n$), we have
\[
\dim (m,n) \cdot L_{s} \geq \frac{k}{2}.
\]
\end{Lemm}

\begin{Lemm} \label{NormKE} 
Let $s \in \R^n$. Then the pair $(L_{s},V \times V^c)$ is $K$-rotund.
\end{Lemm}

The proof of the Main Lemma of Section~\ref{section:green-model} from the analogous lemmas (see the end of Subsection~\ref{subsection:ProofMainLemma}) also works word by word in the new case.



The next lemma follows from the Main Lemma by the same argument as in Section~\ref{section:green-model}.

\begin{Lemm} \label{ImplicitE}
Suppose $(s^0,t^0)$ is an $\RR$-generic point of $\X$. There is a continuous $\RR$-definable function $s \mapsto t(s)$ defined on a neighbourhood $S\subset \R^n$ of $s^0$ and taking values in $\R^n$ such that for all $s \in S$, the point $y(s) := \expE(\eps t(s) + s)$ is in $V$.
\end{Lemm}

Finally, also the proof of Proposition~\ref{prop:elliptic2} from the lemma above is the same as the corresponding proof in Section~\ref{section:green-model}, this time using the density of $\Z+\alpha\Z$ in $\R$, instead of that of the subgroup $Q$.


The question of whether the model on $E$ is $\omega$-saturated is open. Let us simply note that the remarks in Subsection~\ref{sec:omegasat} can be easily adapted to the elliptic curve case case.


\bibliography{jd-bib}
\end{document}